\documentclass{article}

\usepackage[english]{babel}
\usepackage[utf8]{inputenc}
\usepackage[T1]{fontenc}

\usepackage{amsmath,amssymb,amsthm,mathtools}
\usepackage{textcomp}
\usepackage[makeroom]{cancel}
\usepackage{multicol}
\usepackage{colortbl}
\usepackage{xcolor}
\usepackage{tabularx}
\usepackage{tikz}
\usepackage{ragged2e}
\usepackage{caption}
\usepackage{listings}
\usepackage[toc,page]{appendix}
\usepackage{booktabs}
\usepackage{graphicx}
\usepackage{multirow}
\usepackage{dutchcal}
\DeclareMathAlphabet{\pazocal}{OMS}{zplm}{m}{n}

\def\Acal{\pazocal{A}}
\def\Bcal{\pazocal{B}}
\def\Dcal{\pazocal{D}}

\def\Lcal{\pazocal{L}}
\def\Nbb{\mathbb{N}}
\def\Pcal{\pazocal{P}}
\def\Rcal{\pazocal{R}}
\def\Scal{\pazocal{S}}
\DeclareMathOperator{\End}{End}

\usepackage{accents}
\newcommand{\dcap}{\accentset{\frown}{\delta}}
\newcommand{\Dcap}{\accentset{\frown}{\Dcal}}

\newtheorem{thm}{Theorem}[section]

\newtheorem{Lem}[thm]{Lemma}

\newtheorem{Definition}[thm]{Definition}
\newtheorem{rem}[thm]{Remark}
\newtheorem{Example}[thm]{Example}

\usepackage{hyperref}
\usepackage{doi}

\title{On Derivations of Tensor Products of Perm Algebras and Associative Algebras}

\author{
	Jos\'e Gregorio Rodr\'iguez-Nieto\thanks{Departamento de Matem\'aticas, Universidad Nacional de Colombia, Carrera 65 No.~59A-110, Medell\'in, Colombia.}\\[-2mm]
	{\small\textcolor{gray}{\texttt{jgrodrig@unal.edu.co}}}
	\and
	Olga Patricia Salazar-D\'iaz\footnotemark[1]\\[-2mm]
	{\small\textcolor{gray}{\texttt{olgasalazd@unal.edu.co}}}
	\and
	Andr\'es Sarrazola-Alzate\thanks{Departamento de Ciencias B\'asicas e Ingenier\'ia, Universidad EIA, Calle 23 AA Sur No.~5-200, Kil\'ometro 2+200 Variante al Aeropuerto Jos\'e Mar\'ia C\'ordova, Envigado, Antioquia, Colombia.}\\[-2mm]
	{\small\textcolor{gray}{\texttt{andres.sarrazola@eia.edu.co}}}
	\and
	Ra\'ul Vel\'asquez\thanks{Instituto de Matem\'aticas, Universidad de Antioquia, Calle 67 No.~53--108, Medell\'in, Colombia.}\\[-2mm]
	{\small\textcolor{gray}{\texttt{raul.velasquez@udea.edu.co}}}
}
\date{}

\begin{document}
	
	\maketitle
	
	\begin{abstract}
		The study of derivations and their generalizations on non-associative algebras has proven to be fundamental in understanding the internal symmetries and algebraic dynamics of such structures. In this paper, we investigate derivations and diderivations of tensor product dialgebras arising from the combination of a perm algebra and a unital associative algebra. We provide decomposition theorems that characterize these operators in terms of derivations of the individual factors and suitable multiplication maps. Explicit coordinate formulas are also derived, allowing concrete descriptions of the action of derivations and diderivations with respect to natural bases. These results extend classical decomposition theorems for tensor products beyond the associative setting, highlighting the interplay between perm algebras and non-associative algebraic frameworks.
	\end{abstract}
	

	\section*{Introduction}
	
	The study of \emph{derivations} has long been central to algebra, serving as tools for understanding the symmetry and infinitesimal structure of algebraic systems. A derivation on an algebra \(\Acal\) is a linear operator $d$ that satisfies the Leibniz rule:
	\[
	d(xy) = d(x)y + xd(y),
	\]
	for all \(x, y \in \Acal\). The collection of all derivations, \(\Dcal er(\Acal)\), naturally carries a Lie algebra structure via the commutator bracket \cite{DerTensorProductsBresar,Azam2005}.
	\justify
	When constructing new algebraic structures via \emph{tensor products}, one naturally wonders how derivations behave under such operations. In the associative or Lie algebra setting, Saeid Azam generalized classical results by characterizing the derivation algebra of the tensor product of two algebras and its fixed-point subalgebras under automorphisms \cite{Azam2005}. For \emph{non-associative algebras}, Matej Brešar established a decomposition result: assuming mild finiteness conditions, every derivation of a tensor product \(\Rcal \otimes \Scal\) can be expressed as a sum of three types - inner derivations from the nucleus, derivations stemming from one factor paired with central elements of the other, and vice versa \cite{DerTensorProductsBresar}.
	\justify
	Expanding further, recent work on \emph{generalized derivations} of tensor products shows that these mappings in \(\Acal \otimes \Bcal\) are composed of left multiplications by elements in the nucleus and derivations from each factor accompanied by elements in the respective centers \cite{GeneralizedDerivations2021}. These results collectively form a comprehensive algebraic framework for analyzing how derivations and their generalizations interact with tensor product structures beyond the associative realm.
	\justify
	Parallel to these developments, \emph{perm algebras} - which are associative and non commutative structures satisfying permutation identities - have garnered attention through cohomological studies, especially regarding extensions and bialgebraic structures \cite{Hou2023}.
	
	\subsection*{Motivation and Objectives}
	\justify
	Motivated by recent progress on the study of derivations and their behavior under tensor operations, this work focuses on the analysis of \emph{derivations and diderivations} arising from the tensor product between a perm algebra and a unital associative algebra. 
	More precisely, we consider the perm algebra $\Pcal$, defined over $k[x]\otimes k[x]$ with the product $P\otimes f\circ Q\otimes g = P(0)Q(x)\otimes fg$, and investigate the tensor product $\Dcal = \Pcal \otimes \Acal$.
	\justify
	The first goal of this study is to describe how the derivation algebra of $\Dcal$ decomposes in terms of derivations of $\Pcal$ and $\Acal$, together with appropriate multiplication operators that capture their interaction. 
	Building upon this decomposition, we establish an analogous result for diderivations, which act as left derivations on the perm-algebra component and as derivations on the associative component. 
	Finally, we provide explicit coordinate expressions with respect to natural bases, making the internal structure of these operators transparent and clarifying how they act in concrete algebraic terms.
	\justify
	Altogether, these results extend both classical and modern developments in the framework of \emph{dialgebras}—algebraic systems endowed with two associative operations—and offer new insight into how the algebraic components interact within tensor constructions. 
	The approach emphasizes not only the structural decomposition of derivations and diderivations but also their computational and conceptual implications for broader algebraic settings.
	
	\subsection*{Outline}
	
	The paper begins by establishing the necessary background in algebraic structures, introducing perm algebras, dialgebra structures, and the concepts of derivations and diderivations, along with the tensor product construction that underlies our setting. It then develops and proves the main decomposition theorem for derivations of the tensor product dialgebra, relating them to derivations of each factor combined with multiplication operators and drawing analogies with classical decomposition results in the literature. This is followed by an analogous treatment for diderivations, highlighting their dual role as right derivations on the perm component and derivations on the algebra component. The subsequent section provides explicit coordinate formulas with respect to natural bases, making the results more concrete and applicable. Finally, the paper concludes with a discussion of the implications of these decompositions, possible generalizations to operadic and other non-associative frameworks, and prospective directions for further research.
	
	\subsection*{Notation and Conventions}
	
	Throughout this paper, we work over a base field $k$ of characteristic zero, unless stated otherwise. 
	To keep notation consistent, we adopt the following conventions. 
	For a $k$-algebra $\Acal$, the space of its derivations is denoted by $\Dcal er(\Acal)$, while for a dialgebra $\Dcal$, we write $\Dcal ider(\Dcal)$ for the space of its diderivations. 
	All tensor products are taken over the base field $k$.
	\justify
	Elements of $k[x]$ are represented by symbols such as $P(x)$, $Q(x)$, $f(x)$, and $g(x)$. 
	The canonical basis of $k[x]\otimes k[x]$ is written as 
	\[
	x^{\underline{i}} = x^{i_1}\otimes x^{i_2}, \qquad \underline{i}=(i_1,i_2)\in\mathbb{N}^2.
	\]
	Whenever coordinate maps are considered with respect to a basis, they are assumed to vanish for all but finitely many indices. 
	These conventions will be used consistently throughout the text to simplify the exposition and ensure uniformity in notation.
	
	\section{Preliminaries}\label{sec:preliminaries}
	
	In this section, we recall the main algebraic structures involved: left and right perm algebras, dialgebras, and the concepts of derivations and diderivations in the dialgebraic context. We include foundational references from Chapoton, Loday, Hou, and recent work by Restrepo-S\'anchez and collaborators \cite{Restrepo2025}.
	
	\subsection{Left and Right Perm Algebras}
	
	A \emph{left perm algebra} $(\Pcal,\circ)$ is an associative $k$-algebra whose product satisfies the \emph{left permutative identity}:
	\begin{equation}\label{rightperm}
		(x\circ y)\circ z = (y\circ x)\circ z,
	\end{equation}
	for all $x,y,z\in\Pcal$. 
	\justify
	Similarly, a \emph{right perm algebra} satisfies:
	\begin{equation}\label{leftperm}
		x\circ (y\circ z) = x\circ (z\circ y),
	\end{equation}
	for all $x,y,z\in\Pcal$.
	\justify
	These structures are operadically governed by the \(\mathrm{Perm}\) operad, introduced by Chapoton and studied further by Hou~\cite{Hou2023} from the perspective of extension and cohomology theory. The operad \(\mathrm{Perm}\) is Koszul dual to the pre-Lie operad, as explained in Gnedbaye’s operadic treatment~\cite{Gnedbaye2016}.
	
	\begin{Example}\label{Example P_0}
		Let $k[x]$ be the polynomial algebra. Define
		\[
		P(x)\circ Q(x) := P(0)\,Q(x).
		\]
		Then $(k[x],\circ)$ is a left perm algebra: swapping $x$ and $y$ in \eqref{leftperm} does not affect the result because evaluation at zero is symmetric.
	\end{Example}
	
	\subsection{The Dialgebra Variety}
	
	A \emph{dialgebra} $(\Dcal,\vdash,\dashv)$ over $k$ is a $k$-vector space with associative bilinear operations
	\[
	\vdash,\;\dashv : \Dcal\times \Dcal\to \Dcal,
	\]
	which satisfy the compatibility identities
	\begin{equation*}\label{D3}
		x\dashv (y \dashv z)=x\dashv (y \vdash z),
	\end{equation*}
	\begin{equation*}\label{D4}
		(x\dashv y )\vdash z=(x\vdash y) \vdash z,
	\end{equation*}
	\begin{equation*}\label{D5}
		x\vdash (y \dashv z)=(x\vdash y) \dashv z.
	\end{equation*}
	The general theory of dialgebras - also known as “Loday-type algebras” - was introduced by Loday and further developed in operadic terms (see, e.g.,~\cite{Gnedbaye2016}).
	\justify
	Salazar-Díaz and collaborators have contributed significantly to the study of dialgebra varieties. Notably, Salazar-Díaz, Velásquez \& Wills-Toro~\cite{SalazarVelasquezWills2015} propose a construction of dialgebra varieties via bimodules and equivariant mappings, closely related to the Kolesnikov-Pozhidaev construction, which we explain in the next subsection.
	
	\subsubsection{The KP-algorithm and Dialgebras}
	\justify
	Kolesnikov's work \cite{kolesnikov2008}, together with Pozhidaev's extension to $n$-ary algebras \cite{pozhidaev2009}, provided a general algorithm—the \emph{KP-algorithm}—to convert a set of multilinear identities for algebras into a set of multilinear identities for dialgebras.
	\justify
	The idea is to define a dialgebra associated with a given variety $\mathcal{M}$ of algebras such that the corresponding variety of dialgebras coincides with the Hadamard product
	\[
	\Pcal erm \otimes \mathcal{A} ,
	\]
	where $\mathcal{A}$ is the operad governing the variety $\mathcal{M}$ and $\Pcal erm$ is the operad governing associative algebras satisfying the left-commutativity relation
	\[
	(xy)z - (yx)z = 0.
	\]
	\justify
	For an arbitrary variety $\mathcal{M}$ of algebras with one binary operation (governed by an operad $\mathcal{A}$), the KP-algorithm allows one to deduce the defining identities for the class of (\textit{di}-)algebras governed by the operad $\Pcal erm \otimes \mathcal{A}$ starting from the defining identities of $\mathcal{M}$.
	\justify
	The basic idea is the following: take an algebra $\Acal$ in the variety $\mathcal{M}$ with a bilinear operation $*$, and let $(\Pcal,\circ) \in \Pcal erm$. Then the space $\Pcal \otimes \Acal$ equipped with the products $\dashv$ and $\vdash$ defined by
	\begin{equation*}
		(p \otimes a) \vdash (q \otimes b) = p\circ q \otimes (a * b), 
		\qquad
		(p \otimes a) \dashv (q \otimes b) = q\circ p \otimes (a * b),
	\end{equation*}
	for all $p,q \in \Pcal$ and $a,b \in \Acal$, it carries a dialgebra structure combining the perm algebra and the algebra $\Acal$.
	\justify
	Hence, for the variety $\mathcal{M}$, the corresponding variety of dialgebras $di\text{-}\mathcal{M}$ is the variety of all algebras with bilinear operations $\dashv$ and $\vdash$, determined by the identities that hold in all algebras of the form $\Pcal \otimes \Acal$, with $\Acal \in \mathcal{M}$ and $\Pcal \in \Pcal erm$, i.e.
	\[
	di\text{-}\mathcal{M} = \mathrm{Var}\bigl(\{\Pcal \otimes \Acal \mid \Pcal \in \Pcal,\, \Acal \in \mathcal{M}\}\bigr).
	\]
	\justify
	Pavel S.~Kolesnikov and V.~Yu.~Voronin, \cite{Kolesnikov2013}, showed that for every dialgebra $\Dcal \in \mathrm{di}\text{-}\mathcal{M}$ there exists an algebra $\Dcap \in \mathcal{M}$ such that $\Dcal$ embeds into
	\[
	\Pcal_{0} \otimes \Dcap \in di\text{-}\mathcal{M},
	\]
	where $\Pcal_{0} \in \Pcal erm$ is as in example \ref{Example P_0}.
	\justify
	More recently, Restrepo-Sánchez et al.~\cite{Restrepo2025} studied the dialgebra structure on $k[x]\otimes k[x]$ and classified its derivations and diderivations.
	
	\subsection{Derivations and Diderivations on Dialgebras}
	
	Let $(\Dcal,\vdash,\dashv)$ be a dialgebra.
	\justify
	A \emph{derivation} $d: \Dcal\to \Dcal$ is a linear map satisfying:
	\[
	d(x\vdash y) = d(x)\vdash y + x\vdash d(y),\quad
	d(x\dashv y) = d(x)\dashv y + x\dashv d(y),
	\]
	for all $x,y\in \Dcal$. That is, $d$ is simultaneously a derivation for both operations.
	\justify
	We will denote by $\pazocal{D}er(\Dcal)$ the vector space generated by all the derivations $d:\Dcal\rightarrow\Dcal$ defined on $\Dcal$. In order to show the existence of derivations over any dialgebra, we may consider the linear map $ad_a:\Dcal\rightarrow\Dcal$ defined by $ad_a\coloneqq R_a^{\dashv}-L_{a}^{\vdash}$, where $R_a^{\dashv}$ and $L_a^{\vdash}$ are the \textbf{multiplicative operators} $R_a^{\dashv}(b)=b\dashv a$ and $L_a^{\vdash}(b)=a\vdash b$, respectively.
	\justify
	A \emph{diderivation} $\delta: \Dcal\to \Dcal$ satisfies a “mixed” Leibniz property:
	\[
	\delta(x\vdash y) =\delta(x\dashv y) = \delta(x)\dashv y + x\vdash \delta(y).
	\]
	Thus, $\delta$ acts as a derivation on the point side of the products.
	\justify
	We will denote by $\pazocal{D}ider(\Dcal)$ the vector space generated by all the diderivations of $\mathcal{D}$. This space is not empty because for every $a\in\mathcal{D}$ a straightforward calculation shows that $Ad_a\in\pazocal{D}ider(\mathcal{D})$, where $Ad_a=R^{\vdash}_a-L^{\dashv}_a$ and therefore $Ad_a(b)=b\vdash a-a\dashv b$.
	
	\justify
	Moreover, the vector space $\pazocal{D}ider(\Dcal)$ is a bimodule over the de Lie algebra  $\pazocal{D}er(\Dcal)$ with action definied by  $d\cdot \delta:=[d,\delta]$.
	\justify
	The notion of diderivation was employed in the classification of derivations and diderivations of tensor dialgebra structures, as in~\cite{Restrepo2025}.
	
	\begin{rem}
		If $\vdash = \dashv$ coincides with an associative multiplication, then derivations and diderivations coincide.
	\end{rem}
	\justify
	In the setting of perm algebras, we have the following notions. 
	
	\begin{Definition}
		Let $(\Pcal,\cdot)$ be a perm algebra over a field $k$. 
		The vector space consisting of all left derivations on $\Pcal$ is denoted by
		\[
		\Lcal\Dcal er(\Pcal) = \{\, \delta:\Pcal\to\Pcal \mid \delta(x\cdot y)=x\cdot \delta(y)+y\cdot \delta(x),\ \forall\,x,y\in\Pcal \,\}.
		\]
		Elements of $\Lcal\Dcal er(\Pcal)$ are called \emph{left derivations} of $\Pcal$.
	\end{Definition}
	
	\begin{rem}
		In a similar way, one can consider the space of all \emph{right derivations} on $\Pcal$, denoted by $\Rcal\Dcal er(\Pcal)$, consisting of all linear maps $\delta:\Pcal\to\Pcal$ satisfying
		\[
		\delta(x\cdot y)=\delta(x)\cdot y+\delta(y)\cdot x, \qquad \text{for all } x,y\in\Pcal.
		\]
		Both $\Lcal\Dcal er(\Pcal)$ and $\Rcal\Dcal er(\Pcal)$ are linear subspaces of $\End(\Pcal)$. 
		These two families of derivations play complementary roles in the structure theory of perm algebras and their associated dialgebras.
	\end{rem}

	\section{Derivations on the Tensor Product Dialgebra}\label{sec:der_tensor}
	\justify
	In the previous section, we considered the following structure of a (left) perm algebra over the ring of polynomials in a single variable $k[x]$:
	\begin{equation*}
		P(x)\circ Q(x)\coloneqq P(0)Q(x).
	\end{equation*}
	We denote the perm algebra $(K[x],\circ)$ by $\Pcal_0$. In the sequel, we will also consider the tensor product (left) perm algebra $\Pcal\coloneqq k[x]\otimes k[x]$, whose algebraic structure is defined by 
	\begin{equation*}
		P(x)\otimes f(x)\bullet Q(x)\otimes g(x)\coloneqq P(x)\circ Q(x)\otimes f(x)g(x).
	\end{equation*}
	
	\begin{rem}
		Any polynomial with independent coefficient equals one $a_{n}x^n+\cdots + a_1x+1$ is a left unit of both  $\Pcal_0$ and $\Pcal$. In particular, the constant polynomial $1$ is a left unit.
	\end{rem}
	\justify
	On the other hand, let $\Acal$ be an associative unital $K$-algebra.  We may endow the tensor product $\Dcal_0\coloneqq \Pcal_0\otimes\Acal$ of a dialgebra (associative) structure by defining the following products:
	\begin{equation*}
		P(x)\otimes a\vdash Q(x)\otimes b\coloneqq P(x)\circ Q(x)\otimes ab = P(0)Q(x)\otimes ab
	\end{equation*}
	and 
	\begin{equation*}
		P(x)\otimes a\dashv Q(x)\otimes b\coloneqq Q(x)\circ P(x)\otimes ab = Q(0)P(x)\otimes ab.
	\end{equation*}
	We embed the previous dialgebra into the dialgebra $\Dcal\coloneqq \Pcal\otimes \Acal$ whose algebraic structures are defined by:
	\begin{equation*}
		P(x)\otimes f(x)\otimes a\vdash Q(x)\otimes g(x)\otimes b \coloneqq P(x)\circ Q(x)\otimes f(x)g(x)\otimes ab        
	\end{equation*}
	and 
	\begin{equation*}
		P(x)\otimes f(x)\otimes a\dashv Q(x)\otimes g(x)\otimes b \coloneqq Q(x)\circ P(x)\otimes f(x)g(x)\otimes ab.
	\end{equation*}
	\justify
	On the other hand, by definition, every derivation $d\in \Dcal er(\Dcal)$ satisfies 
	\begin{equation}\label{Der_general}
		d(\underline{P}\otimes a) = d(\underline{1}\otimes a\vdash \underline{P}\otimes 1) = d(\underline{1}\otimes a)\vdash \underline{P}\otimes 1 + \underline{1}\otimes a\vdash d(\underline{P}\otimes 1),
	\end{equation}
	this means that, in order to establish a concrete classification of $d$, it is necessary to explicitly determine its behaviour on the basic tensor products $\underline{1}\otimes a$ and $\underline{P}\otimes 1$, where $\underline{P}\in\Pcal$ and $a\in\Acal$.

	We have the following analogue of \cite[Lemma 2.1]{DerTensorProductsBresar}.
	\begin{Lem}\label{Der_Perm}
		Let $\{u_k\;|\; k\in I\}$ be a basis of $\Acal$ as a $k$-vector space and $d\in \Dcal er(\Dcal)$ be a derivation. Then for each $(\underline{i},k)\in \Nbb^2\times I$ there exists a derivation $d_{(\underline{i},k)}$ of $\Pcal$ such that for every $\underline{P}\coloneqq P(x)\otimes f(x)\in \Pcal$ we have
		\begin{equation*}
			d\left(\underline{P}\otimes 1\right) =\displaystyle\sum_{(\underline{i},k)\in \Nbb^2\times I}d_{(\underline{i},k)}\left(\underline{P}\right)\otimes u_k
		\end{equation*}
		and $d_{(\underline{i},k)}(\underline{P})=0$ for all but finitely many $(\underline{i},k)\in \Nbb^2\times I$.
	\end{Lem}
	
	\begin{proof}
		Let us consider the coordinate functions defined by the basis $\{u_k\;|\;k\in I\}$ of $\Acal$ and the monomials  $\{x^{\underline{i}}\; |\; \underline{i}\in\Nbb^2\}$:
		\begin{equation*}
			d(\underline{P}\otimes 1) = \displaystyle\sum_{(\underline{i},k)\in\Nbb^2\times I} \lambda_{(\underline{i},k)}x^{\underline{i}}\otimes u_k = \displaystyle\sum_{(\underline{i},k)\in \Nbb^2\times I}d_{(\underline{i},k)}(\underline{P})\otimes u_k,
		\end{equation*}
		where $\lambda_{(\underline{i},k)}=0$ for all but finitely many $(\underline{i},k)\in \Nbb^2\times I$. In other words, we define $d_{(\underline{i},k)}(\underline{P}) \coloneqq \lambda_{(\underline{i},k)}x^{\underline{i}}$. It is clear that $d_{(\underline{i},k)}:\Pcal \rightarrow \Pcal$ is linear due to the linearity of $d$. Furthermore, if $\underline{Q}=Q(x)\otimes g(x)\in \Pcal$, then 
		\begin{align*}
			\displaystyle\sum_{(\underline{i},k)\in \Nbb^2\times I}d_{(\underline{i},k)}\left(\underline{P}\bullet \underline{Q}\right)\otimes u_k &=
			d\left(\underline{P}\bullet \underline{Q}\otimes 1\right)\\
			& = d((\underline{P}\otimes 1)\vdash (\underline{Q}\otimes 1)) \\
			& = d(\underline{P}\otimes 1)\vdash (\underline{Q}\otimes 1) + (\underline{P}\otimes 1)\vdash d(\underline{Q}\otimes 1)\\
			& = \left(\displaystyle\sum_{(\underline{i},k)\in \Nbb^2\times I}d_{(\underline{i},k)}(\underline{P})\otimes u_k \right)\vdash (\underline{Q}\otimes 1)\\
			&\hspace{0.4 cm} + (\underline{P}\otimes 1)\vdash \left(\sum_{(\underline{i},k)\in \Nbb^2\times I}d_{(\underline{i},k)}(\underline{Q})\otimes u_k\right)\\
			& = \sum_{(\underline{i},k)\in \Nbb^2\times I}\left(d_{(\underline{i},k)}(\underline{P})\bullet \underline{Q}+\underline{P}\bullet d_{(\underline{i},k)}(\underline{Q})\right)\otimes u_{k}
		\end{align*}
		yields
		\begin{equation*}
			\sum_{(\underline{i},k)\in \Nbb^2\times I}\left(d_{(\underline{i},k)}(\underline{P}\bullet \underline{Q})-d_{(\underline{i},k)}(\underline{P})\bullet \underline{Q}-\underline{P}\bullet d_{(\underline{i},k)}(\underline{Q}\right)\otimes u_k = 0,
		\end{equation*}
		which implies $d_{(\underline{i},k)}\in\Dcal er(\Pcal)$.
	\end{proof}
	
	Let us now examine the first term on the right-hand side of (\ref{Der_general}).
	
	\begin{Lem}\label{Der_Alg}
		Let $\left\{x^{\underline{i}}\coloneqq x^{i_1}\otimes x^{i_2}\;|\; \underline{i}\coloneqq \left(i_1,i_2\right)\in \Nbb^2\right\}$ be the canonical basis of $k[x]\otimes k[x]$ as a $k$-vector space, and $d\in \Dcal er(\Dcal)$ be a derivation. Then for every $a\in\Acal$ there exist $J\subseteq \Nbb\times\Nbb$, $\varphi_{(\underline{j},k)}\in Hom_{k-Vect}(\Acal,\Acal)$ and $D_{(0,j,k)}\in \Dcal er(\Acal)$ for every $(\underline{j},k)\in J\times I$ and $(0,j,k)\in \Nbb^2\times I$, such that 
		\begin{equation*}
			d(\underline{1}\otimes a) = \displaystyle\sum_{(\underline{j},k)\in J\times I}x^{\underline{j}}\otimes \varphi_{(\underline{j},k)}(a) + \displaystyle\sum_{(0,j,k)\in \Nbb^2\times I} 1\otimes x^{\underline{j}}\otimes D_{(0,j,k)}(a),
		\end{equation*}
		and $\varphi_{(\underline{j},k)}(a) = D_{(0,j,k)}(a)=0$ for all but finitely many $(\underline{j},k)\in J\times \Nbb$. 
	\end{Lem}
	
	\begin{proof}
		Let us consider the coordinate functions of the derivation $d$ defined by the basis $\left\{x^{\underline{i}}\;|\; \underline{i}\coloneqq \left(i_1,i_2\right)\in \Nbb^2\right\}$ of $\Pcal$ and $\{u_k\;|\; k\in I\}$ of $\Acal$:
		\begin{equation*}
			d(\underline{1}\otimes a)=\displaystyle\sum_{(\underline{i},k)\in\Nbb^2\times I}x^{\underline{i}}\otimes\mu_{(\underline{i},k)}u_k =\displaystyle\sum_{(\underline{i},k)\in \Nbb^2\times I} x^{\underline{i}}\otimes d_{(\underline{i},k)}(a),
		\end{equation*}
		where $\mu_{(\underline{i},k)}\in k$. By linearity $d_{(\underline{i},k)}\in Hom_{k-Vect}(\Acal,\Acal)$ for all $(\underline{i},k)\in \Nbb^2\times I$. Let $b\in\Acal$, then
		\begin{align*}
			d(\underline{1}\otimes a b) & = d(\underline{1}\otimes a\vdash \underline{1}\otimes b) \\
			& = d(\underline{1}\otimes a)\vdash \underline{1}\otimes b + \underline{1}\otimes a\vdash d(\underline{1}\otimes b)\\
			& = \left(\displaystyle\sum_{(\underline{i},k)\in \Nbb^2\times I} x^{\underline{i}}\otimes d_{(\underline{i},k)}(a)\right)\vdash \underline{1}\otimes b \\
			& \hspace{0.4 cm} + \underline{1}\otimes a\vdash \left(\displaystyle\sum_{(\underline{i},k)\in \Nbb^2\times I} x^{\underline{i}}\otimes d_{(\underline{i},k)}(b)\right)\\
			& = \displaystyle\sum_{(\underline{i},k)\in \Nbb^2\times I} x^{\underline{i}}\bullet \underline{1}\otimes d_{(\underline{i},k)}(a)b \; +\; \displaystyle\sum_{(\underline{i},k)\in \Nbb^2\times I} \underline{1}\bullet x^{\underline{i}}\otimes ad_{(\underline{i},k)}(b). 
		\end{align*}
		Let us analyse the terms $x^{\underline{i}}\bullet \underline{1}$ in the first sum. We have:
		\begin{equation*}
			x^{\underline{i}}\bullet \underline{1} = x^{i_1}\otimes x^{i_2}\bullet 1\otimes 1 = x(0)^{i_{1}}\otimes x^{i_2}
			= \left\{ \begin{array}{lcc} 0 & \text{if} & i_{1}\neq 0, \\ \\ \lambda_{0,i_2} & \text{ if } & i_1 = 0.  \end{array} \right.
		\end{equation*}
		Picking up from the previous equation, we obtain:
		\begin{align*}
			d(\underline{1}\otimes ab) & = \displaystyle\sum_{(0,j,k)\in\Nbb^2\times I}\left(1\otimes x^{j}\otimes d_{(0,j,k)}(a)b \; +\; 1\otimes x^{j}\otimes ad_{(0,j,k)}(b)  \right) \\ 
			& \hspace{0.4 cm} + \left(\displaystyle\sum_{(\underline{j},k)\in J\times I} x^{\underline{j}}\otimes \varphi_{(\underline{j},k)}(b)\right),
		\end{align*}
		where $J\subseteq\Nbb^2$ and $\underline{j}=(j_1,j_2)\in J$ satisfy $j_1>0$. Arguing in the same way as in lemma \ref{Der_Perm} we get $D_{(0,j,k)}\coloneqq d_{(0,j,k)}\in \Dcal er(\Acal)$ for all $(0,j,k)\in \Nbb^2\times I$, and therefore
		\begin{equation*}
			d(1\otimes a) = \displaystyle\sum_{(0,j,k)\in\Nbb^2\times I} 1\otimes x^{j}\otimes D_{(0,j,k)}(a) + \left(\displaystyle\sum_{(\underline{j},k)\in J\times I} x^{\underline{j}}\otimes \varphi_{(\underline{j},k)}(a)\right).
		\end{equation*}
	\end{proof}
	
	
	
	\begin{thm}
		Let $\Acal$ be a unital associative algebra. Every derivation $d$ of $\Dcal$ can be written as 
		\begin{equation*}
			d = \sum_{(0,j,k)\in\Nbb^2\times I}id_{k[x]}\otimes L_{ x^{j}}\otimes D_{(0,j,k)}+\sum_{(\underline{i},k)\in\Nbb^2\times I} d_{(\underline{i},k)}\otimes R_{u_k},
		\end{equation*}
		where $d_{(\underline{i},k)}\in\Dcal er(\Pcal)$ for all $(\underline{i},k)\in \Nbb^2\times I$, $D_{(0,j,k)}\in\Dcal er(\Acal)$ for all $j\in\Nbb$, $k\in I$, and $D_{(0,j,k)}=d_{(\underline{i},k)}=0$ for all but finitely many $(\underline{i},k)\in \Nbb^2\times I$ and $(j,k)\in\Nbb\times I$.
	\end{thm}
	
	\begin{proof}
		Let $\underline{P}\coloneqq P(x)\otimes f(x)\in \Pcal$ and $a\in\Acal$. By Lemmas \ref{Der_Perm} and \ref{Der_Alg}, there exist $d_{(\underline{i},k)}\in \Dcal er(\Pcal)$, $(\underline{i},k)\in\Nbb^2\times I$,  $D_{(0,j,k)}\in \Dcal er(\Acal)$, and homomorphisms $\varphi_{(\underline{j},k)}\in Hom_{k-Vect}(\Acal, \Acal)$, $\underline{j} = (j_1,j_2)\in J\subseteq \Nbb\times\Nbb$ with $j_1>0$, such that
		\begin{align*}
			d(\underline{P}\otimes a) & = d(1\otimes a\vdash\underline{P}\otimes 1)\\
			& = d(\underline{1}\otimes a)\vdash \underline{P}\otimes 1 + 1\otimes a \vdash d(\underline{P}\otimes 1) \\
			& =  \Biggl(\displaystyle\sum_{(0,j,k)\in\Nbb^2\times I} 1\otimes x^{j}\otimes D_{(0,j,k)}(a)\\
			& \hspace{0.8 cm} +\left(\displaystyle\sum_{(\underline{j},k)\in J\times I} x^{\underline{j}}\otimes \varphi_{(\underline{j},k)}(a)\right)\Biggr)\vdash \underline{P}\otimes 1 \\
			& \hspace{0.4 cm} + 1\otimes a\vdash \left(\displaystyle\sum_{(\underline{i},k)\in\Nbb^2\times I}d_{(\underline{i},k)}\left(\underline{P}\right)\otimes u_k\right) \\
			& = \displaystyle\sum_{(0,j,k)\in\Nbb^2\times I} P(x)\otimes x^{j}f(x)\otimes D_{(0,j,k)}(a) + \displaystyle\sum_{(\underline{i},k)\in\Nbb^2\times I}^n d_{(\underline{i},k)}(\underline{P})\otimes a u_{k} \\
			& = \sum_{(0,j,k)\in\Nbb^2\times I}P(x)\otimes L_{ x^{j}}(f(x))\otimes D_{(0,j,k)}(a)\\
			&\hspace{0.4 cm} +\sum_{(\underline{i},k)\in\Nbb^2\times I} d_{(\underline{i},k)}(\underline{P})\otimes R_{u_k}(a).
		\end{align*}
	\end{proof}
	
	\section{Diderivations on the Tensor Product Dialgebra}\label{sec:Dider_tensor}
	\justify
	Similarly to lemma \ref{Der_Perm}, the diderivations can be expressed in terms of coordinate functions on the perm algebra, which behave as right derivations. We reamrk for the reader that if $\delta\in\Dcal ider(\Dcal)$, then
	\begin{equation}
		\delta(\underline{P}\otimes a) = \delta(\underline{1}\otimes a\vdash \underline{P}\otimes 1)
		= \delta(\underline{1}\otimes a)\dashv \left(\underline{P}\otimes 1\right) + \left(\underline{1}\otimes a\right)\vdash \delta(\underline{P}\otimes 1).
	\end{equation}
	Therefore, as in the previous section, we need to study the expressions $\delta(\underline{1}\otimes a)$ and $\delta(\underline{P}\otimes 1)$, for $a\in\Acal$ and $\underline{P}\in\Pcal$.  This will be carried out in the following lemmas.
	
	\begin{Lem}
		Let $\{u_k\;|\; i\in I\}$ be a basis of $\Acal$ as a $k$-vector space and $\delta\in \Dcal ider(\Dcal)$ be a diderivation. Then for each $(\underline{i},k)\in \Nbb^2\times I$ there exist left derivations $\dcap_{(\underline{i},k)}$ of $\Pcal$ such that for every $\underline{P}\coloneqq P(x)\otimes f(x)\in \Pcal$ we have
		\begin{equation*}
			\delta\left(\underline{P}\otimes 1\right) =\displaystyle\sum_{(\underline{i},k)\in \Nbb^2\times I}\dcap_{(\underline{i},k)}\left(\underline{P}\right)\otimes u_k
		\end{equation*}
		and $\dcap_{(\underline{i},k)}(\underline{P})=0$ for all but finitely many $(\underline{i},k)\in \Nbb^2\times I$.
	\end{Lem}
	
	\begin{proof}
		Following the approach used in the proof of Lemma \ref{Der_Perm}, we consider the coordinate functions $\dcap_{(\underline{i},k)}$ defined by the basis $\{u_k\;|\;k\in I\}$ of $\Acal$ and the basis of monomials $\{x^{\underline{i}}\;|\; \underline{i}\in\Nbb^2\}$. Let $\underline{P},\underline{Q}\in\Pcal$. We have:    
		\begin{align*}
			\delta(\underline{P}\bullet\underline{Q}\otimes 1) & = \delta(\underline{P}\otimes 1\vdash \underline{Q}\otimes 1) \\
			& = \delta(\underline{P}\otimes 1)\dashv \underline{Q}\otimes 1 + \underline{P}\otimes 1\vdash \delta(\underline{Q}\otimes 1) \\
			& =\displaystyle\sum_{(\underline{i},k)\in \Nbb^2\times I}\left(\underline{Q}\bullet \dcap_{(\underline{i},k)}(\underline{P}) + \underline{P}\bullet\dcap_{(\underline{i},k)}(\underline{Q})\right)\otimes u_k.
		\end{align*}
		The previous chain of equations shows that $\dcap_{(\underline{i},k)}\in\Lcal\Dcal er(\Pcal)$.
	\end{proof}
	
	\justify
	Now, over $\Acal$, the diderivation products coincide due to $\Acal$ being the algebra component of $\Dcal$. Consequently, the diderivations must behave as derivations when expressed in terms of the coordinate functions determined by the respective basis of $k[x]\otimes k[x]$ and $\Acal$.
	
	\begin{Lem}
		Let $\left\{x^{\underline{i}}\;|\; \underline{i}\coloneqq \left(i_1,i_2\right)\in \Nbb\times \Nbb\right\}$ be the canonical basis of $k[x]\otimes k[x]$ as a $k$-vector space, and $\delta\in \Dcal ider(\Dcal)$ be a derivation. Then, for every $a\in\Acal$, the coordinate functions $\delta_{(\underline{i},k)}\in \Dcal er(\Acal)$.
	\end{Lem}
	
	\begin{proof}
		Let $a,b\in \Acal$. Denoting by $\delta_{(\underline{i},k)}$ the coordinate functions of $\delta$ we have
		\begin{align*}
			\delta(\underline{1}\otimes ab) & = \delta(\underline{1}\otimes a \vdash \underline{1}\otimes b) \\
			& = \delta(\underline{1}\otimes a)\dashv \underline{1}\otimes b + \underline{1}\otimes a\vdash \delta(\underline{1}\otimes b)\\
			& = \left(\displaystyle\sum_{(\underline{i},k)\in \Nbb^2\times I} x^{\underline{i}}\otimes \delta_{(\underline{i},k)}(a)\right)\dashv \underline{1}\otimes b\\
			& \hspace{0.4 cm} +\underline{1}\otimes a\vdash \left(\displaystyle\sum_{(\underline{i},k)\in \Nbb^2\times I} x^{\underline{i}}\otimes \delta_{(\underline{i},k)}(b)\right)\\
			& = \displaystyle\sum_{(\underline{i},k)\in \Nbb^2\times I} x^{\underline{i}}\otimes \left(\delta_{(\underline{i},k)}(a)b + a\delta_{(\underline{i},k)}(b)\right).
		\end{align*}
		This clearly implies $\delta_{(\underline{i},k)}\in \Dcal er(\Acal)$.
	\end{proof}
	\justify
	Altogether, we have the following explicit description of diderivations on $\Dcal$.
	
	\begin{thm}
		Let $\Acal$ be a unital associative algebra. Every diderivation $\delta$ of $\Dcal$ can be written as 
		\begin{equation*}
			\delta = \sum_{(\underline{i},k)\in\Nbb^2\times I}\left( x^{i_1}ev_{0} \otimes L_{x^{i_2}} \otimes \delta_{(\underline{i},k)} +  \dcap_{(\underline{i},k)} \otimes R_{u_k}\right),   
		\end{equation*}
		where $ev_{0}$ is the evaluation morphism at zero, $\delta_{(\underline{i},k)}\in\Dcal er(\Acal)$, $\dcap_{(\underline{i},k)}\in\Rcal\Dcal er(\Pcal)$ and $\delta_{\underline{i},k}=\dcap_{(\underline{i},k)}=0$ for all but finitely many $(\underline{i},k)\in\Nbb^2\times I$.
	\end{thm}
	
	\begin{proof}
		Let $\underline{P}=P(x)\otimes f(x)\in \Pcal$, and $a\in\Acal$. We have:
		\begin{align*}
			\delta(\underline{P}\otimes a) & = \delta(\underline{1}\otimes a\vdash \underline{P}\otimes 1) \\
			& = \delta(\underline{1}\otimes a)\dashv \left(\underline{P}\otimes 1\right) + \left(\underline{1}\otimes a\right)\vdash \delta(\underline{P}\otimes 1)\\
			& = \left(\sum_{(\underline{i},k)\in\Nbb^2\times I}x^{\underline{i}}\otimes \delta_{(\underline{i},k)}(a)\right)\dashv (\underline{P}\otimes 1)\\
			&\hspace{0.4 cm} +(\underline{1}\otimes a)\left(\vdash \displaystyle\sum_{(\underline{i},k)\in\Nbb^2\times  I}\dcap_{(\underline{i},k)}\left(\underline{P}\right)\otimes u_k\right)\\
			& = \sum_{(\underline{i},k)\in\Nbb^2\times I}P(0)x^{i_1}\otimes x^{i_2}f(x)\otimes \delta_{(\underline{i},k)}(a) + \displaystyle\sum_{(\underline{i},k)\in \Nbb^2\times I}\dcap_{(\underline{i},k)}(\underline{P})\otimes a u_k.
		\end{align*}
	\end{proof}
	
	\section{Conclusions and Future Work}
	
	In this work we have studied the structure of derivations and diderivations on tensor product dialgebras constructed from a perm algebra and a unital non-associative algebra.
	The findings enhance our knowledge of the interaction between perm algebras and non-associative algebraic systems by extending classical decomposition theorems of derivations in associative and Lie algebra settings into the more general context of dialgebras.
	
	\subsection*{Future Directions}
	
	The methods and results presented in this work open several promising directions for further research. 
	A natural continuation would be to explore operadic generalizations, which would involve applying the decomposition theorems established here to larger classes of operads associated with perm algebras, such as pre-Lie and dendriform structures. 
	A more cohesive understanding of the derivation theories of these algebraic frameworks may result from this method, which may also uncover deeper structural relationships between them.
	\justify
	From a cohomological point of view, examining the behavior of derivations and diderivations would be intriguing, especially with regard to extension and deformation theory. 
	Gaining insight into these relationships may lead to improved classification methods and new invariants in the homological context.
	\justify
	Finally, another line of research involves classification problems for derivations and diderivations in wider families of tensor product dialgebras, including cases with infinite-dimensional components or additional algebraic constraints. 
	Such studies may uncover new regularities in the behaviour of derivation spaces and highlight structural properties that extend beyond the finite-dimensional context.

\end{document}